\documentclass[12pt]{amsart}

\newcommand{\cstar}{$C^{*}$}

\newcommand{\N}{\mathbb{N}}

\newcommand{\R}{\mathbb{R}}
\newcommand{\Z}{\mathbb{Z}}
\newcommand{\C}{\mathbb{C}}

\newcommand{\F}{\mathcal{F}}
\newcommand{\G}{\mathcal{G}}
\newcommand{\U}{\mathcal{U}}
\renewcommand{\P}{\mathcal{P}}

\newcommand{\uK}{\underline{K}}
\newcommand{\jiangsu}{$\mathcal{Z}$}
\newcommand{\algK}{K_1^{\operatorname{alg}}}

\usepackage{amsmath,amsthm,amssymb,paralist}
\usepackage{mathtools}

\DeclarePairedDelimiter\abs{\lvert}{\rvert}
\DeclarePairedDelimiter\norm{\lVert}{\rVert}
\DeclarePairedDelimiter\floor{\lfloor}{\rfloor}

\makeatletter
\let\oldabs\abs
\def\abs{\@ifstar{\oldabs}{\oldabs*}}
\makeatother

\makeatletter
\let\oldnorm\norm
\def\norm{\@ifstar{\oldnorm}{\oldnorm*}}
\makeatother

\makeatletter
\let\oldfloor\floor
\def\floor{\@ifstar{\oldfloor}{\oldfloor*}}
\makeatother

\DeclareMathOperator{\atom}{at}
\DeclareMathOperator{\Aff}{Aff}

\DeclareMathOperator{\dist}{dist}
\DeclareMathOperator{\ext}{ext}
\DeclareMathOperator{\Hom}{Hom}
\DeclareMathOperator{\sa}{sa}
\DeclareMathOperator{\spec}{sp}
\DeclareMathOperator{\Tr}{Tr}

\newtheorem{lem}{Lemma}[section]
\newtheorem{prp}[lem]{Proposition}
\newtheorem{thm}[lem]{Theorem}

\theoremstyle{definition}
\newtheorem{dfn}[lem]{Definition}
\theoremstyle{remark}

\title[Approximate Diagonalization of Unital Homomorphisms]
{Approximate Diagonalization of Unital Homomorphisms from AH-Algebras to
Certain Simple Classifiable \cstar-Algebras}
\author{Min Yong Ro}
\email{minyongro\-@gmail.com}

\begin{document}

\begin{abstract}
In this paper, we prove that unital homomorphisms from a commutative
\cstar-algebra to matrices over a \cstar-algebra with tracial rank at
most one are approximately diagonalizable. We also consider some
generalizations of this result.
\end{abstract}

\maketitle

\section{Introduction} \label{sec:intro}
One of the fundamental facts in linear algebra is that normal matrices
over $\C$ are unitarily equivalent to diagonal matrices. Given the
importance of matrices over \cstar-algebras, it is a natural question to
ask whether a normal matrices with entries in a \cstar-algebra are
diagonalizable in a similar way.

Richard Kadison demonstrated in \cite{kad-diag} that normal matrices
over a von Neumann algebra are diagonalizable. Kadison also posed the
question: for what topological spaces $X$ is every normal matrix over
$C(X)$ diagonalizable? Karsten Grove and Gert Pedersen answered this
question in \cite{grpd-diag} that $X$ must, among other topological
restrictions, be sub-Stonean. This suggests that diagonalization for
normal matrices is restricted to a particularly special class of
\cstar-algebra. Kadison's proof, for example, is based on type
decomposition and the abundance of projections in maximal abelian
subalgebras of von Neumann algebras. This can be extended to show
diagonalization in \cstar-algebras with similar properties, such as
$AW^{*}$-algebras, as seen in \cite{here-diag}, but does not generalize
to larger classes of \cstar-algebras.

If we instead consider approximate diagonalization, then the situation
improves. For example, Yifeng Xue proved in \cite{xue-diag} that
every self-adjoint matrix over $C(X)$ is approximately diagonalizable
if $\dim(X)\leq 2$ and $\check{H}^2(X,\Z) = 0$. Further, Huaxin Lin
proved in \cite{lin-diag} that if $X$ is locally an absolute retract and
$Y$ has $\dim(Y)\leq 2$, then every unital homomorphism from $C(X)$ to
$M_n(C(Y))$ is approximately diagonalizable. In the noncommutative case,
Shuang Zhang proved in \cite{zh-diag} that self-adjoint matrices over
a \cstar-algebra with real rank zero are approximately diagonalizable.

Recall that if $a$ is a normal element in $M_n(A)$ for some unital
\cstar-algebra $A$, then continuous functional calculus induces a unital
homomorphism $\phi\colon C(\spec(a)) \to M_n(A)$. It is easy to
see that approximate diagonalization of the element $a$ is equivalent to
the approximate diagonalization of the induced homomorphism $\phi$.

We consider a slight generalization of the typical matricial
approximate diagonalization.

\begin{dfn}
Let $C$ and $A$ be unital \cstar-algebras. A unital homomorphism
$\phi\colon C\to A$ is {\em approximate diagonalizable} if for any
$\varepsilon > 0$, a finite set $\F\subseteq C$, a positive integer $n$,
and mutually orthogonal projections $e_1, \dotsc, e_n\in A$, there exist
unital homomorphisms $\phi_i\colon C\to e_iAe_i$ and a unitary $u\in A$
such that
\[
\norm{u\phi(f)u^{*} - \sum_{i=1}^n \phi_i(f)} < \varepsilon
\]
for all $f\in F$.
\end{dfn}

Note that this definition implies than the standard matricial
notion of approximate diagonalization by considering $M_n(A)$ for $A$
and projections $e_{i,i}\otimes 1_A$ for $e_i$, where $e_{i,j}$ denotes
the standard matrix units.

The main result of the paper is the approximate diagonalization of
unital homomorphisms from $C(X)$ to \cstar-algebras of tracial rank at
most one for any compact metric space $X$. To prove this result, we use
the classification of unital monomorphisms from AH-algebras to
\cstar-algebras of tracial rank at most one proved by Lin in
\cite{lin-trone}. In Section \ref{sec:prelim}, we review the invariants
used in Lin's classification theorems. In Section \ref{sec:groups}, we
prove a few lemmas related to the decomposition of ordered group
homomorphisms. In Section \ref{sec:main}, we
prove the main result. Though the classification of monomorphisms holds
for larger classes of domains and codomains, approximate diagonalization
does not hold generally in those cases. We give some limited results of
the approximate diagonalization of other homomorphisms in Section
\ref{sec:other}.

\section{Preliminaries} \label{sec:prelim} 

We use the notation found in \cite{lin-trone} and \cite{lin-niu-cl}. In
particular, if $A$ is a unital \cstar-algebra, let $T(A)$ denote the
space of tracial states of $A$ and $\Aff(T(A))$ as the partially ordered
vector space of continuous affine real-valued maps on $T(A)$. There is a
natural pairing between $K_0(A)$ and $T(A)$, which we describe with a
normalized positive group homomorphism
$\rho_A\colon K_0(A)\to \Aff(T(A))$ defined by
$\rho_A([p]) = \tau\otimes\Tr(p)$ for $p\in M_{\infty}(A)$, where $\Tr$
is the unnormalized trace on $M_{\infty}(\C)$. Given another
unital \cstar-algebra $C$ and a unital homomorphism from $C$ to $A$, by
naturality, a commutative square is induced from this pairing. On the
other hand, a normalized positive group homomorphism
$\alpha\colon K_0(C) \to K_0(A)$ and a unital positive linear map
$\gamma\colon \Aff(T(C)) \to \Aff(T(A))$ are called {\em compatible} if
$\rho_{A}\circ\alpha = \gamma\circ \rho_{C}$.

$KL(C,A)$ is the quotient of $KK(C,A)$ by the group of pure extensions
of $K_{*}(C)$ by $K_{*+1}(A)$. See Section 2.4.8 of \cite{ror-cl} for
details. Since the $KL$-group is a quotient of the $KK$-group we have a
$KL$ version of the UCT (see equation 2.4.9 of \cite{ror-cl}) when $C$
satisfies the UCT:
\[
0 \to \ext(K_{*}(C), K_{*+1}(A)) \xrightarrow{\varepsilon} KL(C,A)
\xrightarrow{\Gamma} \Hom(K_{*}(C),K_{*}(A)) \to 0.
\]
For any unital \cstar-algebra $A$, let
$\uK(A) = \oplus_{k=0}^{\infty} \oplus_{i=0}^{1} K(A;\Z/k)$. By
Dadarlat and Loring (\cite{umct}), if $C$ is a \cstar-algebra
satisfying UCT and $A$ is a $\sigma$-unital \cstar-algebra, then we have
\[
KL(C,A) \cong \Hom_{\Lambda}(\uK(C),\uK(A)),
\]
where the homomorphisms are graded group homomorphisms that preserve
certain Bockstein operations. See \cite{umct} for details. We will
identify $KL(C,A)$ with this group of homomorphisms.

Let $KL_e(C,A)^{++}$ denote the set of $\kappa\in KL(C,A)$ satisfying
$\Gamma(\kappa)(K_{0}(C)^{+}\setminus\{0\}) \subseteq K_0(A)\setminus\{0\}$
and $\Gamma(\kappa)([1_C]_0) = [1_A]_0$. We call $\kappa\in KL_e(C,A)^{++}$
and a unital positive linear map $\gamma\colon \Aff(T(C))\to \Aff(T(A))$ 
{\em compatible} if the restriction of $\Gamma(\kappa)$ to $K_0(C)$ and
$\gamma$ are compatible.

Notice that for any compact metric space $X$, the range of $\rho_{C(X)}$
is isomorphic to $C(X,\Z)$. Consequently, the short-exact sequence:
\[
0 \to \ker\rho_{C(X)} \to K_0(C(X)) \to C(X,\Z) \to 0
\]
is split, since $C(X,\Z)$ is a free abelian group. This is apparent
in the case when $X$ has finitely many connected components, where
$C(X,\Z)$ is generated by the characteristic functions of the connected
components of $X$. Furthermore, we note that
$\Aff(T(C(X))) \cong C(X)_{\sa}$.

Let the group of unitaries of $C$ be denoted by $U(C)$, the normal
subgroup of the connected component containing $1_C$ by $U_0(C)$, the
closed normal subgroup generated by the commutators of $U(C)$ by
$CU(C)$, and $CU_0(C) = CU(C) \cap U_0(C)$. We also define
$U^{\infty}(C) = \cup_{n=1}^{\infty} U(M_n(C))$, and similarly define
$U_0^{\infty}(C)$, $CU^{\infty}(C)$, and $CU^{\infty}_0(C)$. Let
$\algK(C) = U^{\infty}(C)/CU^{\infty}(C)$. For every unitary
$u\in U^{\infty}(C)$, the equivalence class in $\algK(C)$ containing $u$
is denoted by $\bar{u}$.

As seen in \cite{thom-traces}, we have the following short-exact
sequence:
\[
0 \to \Aff(T(C))/\rho_C(K_0(C)) \to \algK(C) \to K_1(C) \to 0
\]
This short-exact sequence is split, though unnaturally. Let $\pi_C$
denote the quotient map $\algK(C) \to K_1(C)$. Given a unital
homomorphism $\phi\colon C\to A$, let the induced continuous
homomorphism be denoted by
$\phi^{\ddagger}\colon \algK(C) \to \algK(A)$.

Suppose $\kappa\in KL_e^{++}(C,A)$ and
$\gamma\colon \Aff(T(C))\to \Aff(T(A))$ are compatible.
Let $\eta\colon \algK(C) \to \algK(A)$ be a continuous homomorphism.
If the restriction of $\eta$ to $\Aff(T(C))/\overline{\rho_C(K_0(C))}$
is equal to the homomorphism induced from $\gamma$ and the restriction
of $\eta$ and $\kappa$ to $K_1(C)$ are equal, then we say that
$\kappa$, $\gamma$, and $\eta$ are {\em compatible}.

We conclude this section with a uniqueness theorem of Lin's:

\begin{thm} \label{thm:trone-unique}
Let $C$ be a unital AH-algebra and let $A$ be a separable simple unital
\cstar-algebra with tracial rank at most one. Let $\phi\colon C\to A$
be a unital monomorphism. For every $\varepsilon > 0$ and
every finite set $\F\subseteq C$, there exist $\delta > 0$, a finite
set $\P\subseteq \uK(C)$, a finite set $\U\subseteq U^{\infty}(C)$, and
a finite set $\G\subseteq C_{\sa}$ such that that for any unital
homomorphism $\psi\colon C\to A$, if
$KL(\phi)(p) = KL(\psi)(p)$ for $p\in \P$,
$\dist(\phi^{\ddagger}(\bar{z}),\psi^{\ddagger}(\bar{z})) < \delta$ for
$z\in \U$,
and $\abs{\tau\circ\phi(g) - \tau\circ\psi(g)} < \delta$ for $g\in\G$,
then there exists a unitary $u\in A$ such that
\[
\norm{u\phi(f)u^{\ast} - \psi(f)} < \varepsilon
\]
for all $f\in \F$.
\end{thm}

This is simply Corollary 11.6 of \cite{lin-auetrone} without the
condition that $C$ has Property (J). The same proof works in light of
Theorem 5.8 and Lemma 5.7(2) of \cite{lin-trone}.

\section{Decomposition of Ordered Group Homomorphisms}
\label{sec:groups}

As a group, $K_0(C(X))$ can be written as the inductive limit of
finitely generated abelian groups and so it is a relatively
straightforward matter to define homomorphisms from $K_0(C(X))$. The
ordering of $K_0(C(X))$ is not easily determined since topological
properties may lead to perforation. Fortunately, if the target of the
homomorphism has an ordering determined by its traces, then these
challenges can be managed, and we can define positive group
homomorphisms. We adopt some language and notation about partially
ordered abelian groups from \cite{goodearl}. 

We define the group homomorphism
$\rho_G\colon G\to \Aff(S(G),1)$ by $\rho_G(g)(\sigma)
= \sigma(g)$. We note that the intersection of the kernel of the traces
of $(G,u)$ is equal to $\ker\rho_G$. Also when $C$ is exact, $\rho_C
= \rho_{K_0(C)}$.

Let $\atom(G)$ denote the subgroup of $G$ generated by its atoms.
When $X$ is topological space with finitely many connected components,
the characteristic functions of those components are the atoms of
$K_0(C(X))^{+}$ and $\atom(K_0(C(X))) \cong C(X,\Z)$. As noted in
Section \ref{sec:prelim}, when $X$ has finitely many connected
components, $K_0(C(X))$ can be decomposed into the direct
sum of $\atom(K_0(C(X)))$ and $\ker\rho_{C(X)}$.

In contrast, when a partially ordered abelian group $G$ is simple,
$G^{+}$ contains no atoms except when $G$ is cyclic (Lemma 14.2 of
\cite{goodearl}). It will be useful to treat $\Z$ separately. For
example, when $G$ is a non-cyclic, simple interpolation group,
$G$ also satisfies a strict version of interpolation (see Proposition
14.6 of \cite{goodearl}).

\begin{dfn}
A partially ordered abelian group $G$ has \emph{strict interpolation} if
for all $x_1$, $x_2$, $y_1$, $y_2 \in G$ such that $x_i < y_j$ for all
$i,j$, there exists $z\in G$ such that $x_i < z < y_j$ for all $i,j$.
\end{dfn}

Strict versions of the Riesz decomposition properties follow with
analogous proofs. See, for example, Propositions 2.1 and 2.2 of
\cite{goodearl}.

\begin{prp}
Let $G$ be a partially ordered abelian group. The following are
equivalent:
\begin{asparaenum}[(a)]
\item $G$ has strict interpolation.

\item If $x,y_1,y_2\in G$ satisfying $0 < x < y_1 + y_2$, then there
exist $x_1,x_2\in G^{+}\setminus\{0\}$ such that $x_1 + x_2 = x$ and
$x_i < y_i$ for $i=1,2$.

\item If $x_1,x_2,y_1,y_2 \in G^{+}\setminus\{0\}$ satisfying $x_1+x_2=
y_1 + y_2$, then there exist $z_{i,j}\in G^{+}\setminus\{0\}$ such that
$x_i = z_{i,1} + z_{i,2}$ and $y_j = z_{1,j} + z_{2,j}$ for $i=1,2$,
$j=1,2$.
\end{asparaenum}
\end{prp}

\begin{prp}
Let $G$ be a partially ordered abelian group with strict interpolation.
Then the following hold:
\begin{asparaenum}[(a)]
\item If $x_1, x_2,\dotsc x_n$ and $y_1,\dotsc, y_k$ are in $G$ such
that $x_i < y_j$ for all $i,j$, then there exists $z\in G$ such that
$x_i < z < y_j$ for all $i,j$.

\item If $x,y_1,y_2,\dotsc,y_n\in G^{+}\setminus\{0\}$ satisfying
$x < y_1 + y_2 + \dotsc y_n$, then there exist
$x_1,\dotsc,x_n\in G^{+}\setminus\{0\}$ such that $x = x_1+\dotsb+x_n$
and $x_i < y_i$ for all $i$.

\item If $x_1,\dotsc,x_n,y_1,\dotsc,y_k\in G^{+}\setminus\{0\}$, then
there exists $z_{i,j}$ for $i=1,2,\dotsc,n$, $j=1,2,\dotsc,k$ such that
$x_i = z_{i,1} + \dotsb + z_{i,k}$ and
$y_j = z_{1,j} + \dotsb + z_{n,j}$.
\end{asparaenum}
\end{prp}

When a partially ordered abelian group $G$ is simple and weakly
unperforated, the order on $G$ is determined by its traces. Namely for
all $x\in G$, $x > 0$ if and only if $\sigma(x) > 0$ for all
$\sigma\in S(G,u)$. 

\begin{lem} \label{lem:gp-base}
Let $G$ be a partially ordered abelian group such that $G^{+}$ has
finitely many atoms $\{x_1,x_2,\dotsc, x_k\}$ and $u = \sum_{j=1}^k x_j$
is an order unit. Suppose $G = \atom(G) \oplus \ker\rho_G$.
Let $n \geq 1$ be an integer and let $H$ be a simple interpolation group
and order units $v_i$ for $i=1,2,\dotsc,n$.
 
For any normalized positive group homomorphism
$\alpha\colon (G,u)\to (H,\sum_{i=1}^{n}v_i)$, there exist normalized
positive group homomorphisms $\alpha_i\colon (G,u)\to (H,v_i)$ for
$i=1,2,\dotsc,n$ such that
$\alpha = \alpha_1 + \alpha_2 + \dotsb + \alpha_n$ and
$\ker\rho_G\subseteq \ker\,\alpha_i$ for $i>1$.

Furthermore, if $H$ has strict interpolation and
$\ker\alpha \cap \atom(G) = 0$, then we can arrange it so that
$\ker\alpha_i\cap \atom(G) = 0$ for all $i$.
\end{lem}

\begin{proof}
Since $\alpha(x_1) + \alpha(x_2) + \dotsb + \alpha(x_k) =
\alpha(u) = v_1 + v_2 + \dotsb + v_n$, by the Riesz interpolation
property, there exist $z_{i,j}\in H^{+}$ for $i=1,2,\dotsc,n$,
$j=1,2,\dotsc,k$ such that
\[
\sum_{i=1}^{n} z_{i,j} = \alpha(x_j) \text{ and }
\sum_{j=1}^{k} z_{i,j} = v_i
\]

We define $\alpha_i\colon G\to H$ by setting $\alpha_i(x_j) = z_{i,j}$
for all $i$ and $j$, setting $\alpha_1(g) = \alpha(g)$ for
$g\in \ker\rho_G$, and setting $\alpha_i(g) = 0$ for $g\in \ker\rho$ and
$i>1$. Since the set of atoms is $\Z$-independent (Lemma 3.10 of
\cite{goodearl}), $\alpha_i$ is a group homomorphism for every $i$. By
construction, $\ker\rho_G\subseteq \ker\alpha_i$ for $i>1$.

Since $\sum_{i=1}^n \alpha_i = \alpha_1 = \alpha$ on $\ker\rho_G$ and
$\sum_{i=1}^n \alpha_i(x_j) = \sum_{i=1}^n z_{i,j} =
\alpha(x_j)$, we have $\sum_{i=1}^n \alpha_i = \alpha$.
Let $x\in G^{+}$. There exist non-negative integers $m_j$ for
$j=1,2,\dotsc,k$ and $g\in \ker\rho$ so that
$x = g + \sum_{j=1}^k m_j x_j$. Take $\tau\in S(G)$. Since
$\tau\circ\alpha_i\in S(G,u)$, we have $\tau(\alpha_i(g)) = 0$ for all
$i$ and so $\tau(\alpha_i(x)) = \sum_{j=1}^n m_j \tau(z_{i,j}) \geq 0$.
So we have that
$\alpha_i(x) \geq 0$, and so $\alpha_i$ are positive group
homomorphisms. Also $\alpha_i(u) = \alpha_i(\sum_{j=1}^k x_j)
= \sum_{j=1}^k \alpha_i(x_j) = \sum_{j=1}^k z_{i,j} = v_i$. So
$\alpha_i\colon (G,u) \to (H,v_i)$ is a normalized positive group
homomorphism for every $i$.

Suppose that $G$ has strict interpolation and $\ker\alpha \cap
\atom(G) = 0$. Then $\alpha(x_j) > 0$ and
since $v_i > 0$, by strict comparison, we can arrange $z_{i,j} > 0$. And
so $\ker\alpha_i \cap \atom(G) = 0$ for all
$i$. 
\end{proof}

\begin{lem} \label{lem:gp-induction}
Let $G_1$ and $G_2$ be partially ordered abelian groups such that
$G_1^{+}$ has finitely many atoms $\{x_1,x_2,\dotsc,x_k\}$ and 
$G_2^{+}$ has finitely many atoms $\{y_1,y_2,\dotsc,y_m\}$, where
$u_1 = \sum_{j=1}^k x_j$ and $u_2 = \sum_{t=1}^m y_t$ are order units.
Suppose that $G_1 = \atom(G_1) \oplus \ker\rho_{G_1}$
and  $G_2 = \atom(G_2) \oplus \ker\rho_{G_2}$.
Let $n \geq 1$ be an integer and let $H$ be a simple interpolation
group with order units $v_i$ for $i=1,2,\dotsc,n$.

Let $\alpha\colon G_1\to G_2$ be a normalized positive
group homomorphism such that
$\alpha(\atom(G_1))
\subseteq \atom(G_2)$. Let
$\beta_s\colon (G_s,u_s)\to (H,\sum_{i=1}^n v_i)$ be normalized positive
group homomorphisms for $s=1,2$ such that
$\beta_1 = \alpha\circ \beta_2$. If there exist
$\beta_{1,i}\colon (G_1,u_1)\to (H,v_i)$ such that
$\sum_{i=1}^n \beta_{1,i} = \beta_1$ and
$\ker\rho_{G_1}\subseteq \ker\beta_{1,i}$ for $i>1$, then there exist
$\beta_{2,i}\colon (G_2,u_2)\to (H,v_i)$ such that
$\sum_{i=1}^n \beta_{2,i} = \beta_2$,
$\ker\rho_{G_2}\subseteq \ker\beta_{2,i}$ for $i>1$, and
$\beta_{1,i} = \beta_{2,i}\circ \alpha$ for all $i$.

Furthermore, if $H$ has strict interpolation,
if $\ker\alpha\cap \atom(G_1) = 0$, $\ker\beta_2 \cap \atom(G_2) = 0$,
and $\ker\beta_{1,i} \cap \atom(G_1) = 0$ for all $i$,
then we can arrange it so that
$\ker\beta_{2,i} \cap \atom(G_2) = 0$ for all $i$.
\end{lem}

\begin{proof}
Since $\alpha$ is a positive homomorphism, $\alpha(u_1) = u_2$, and
$\alpha(\atom(G_1)) \subseteq \atom(G_2)$, for each $j=1,2,\dotsc,k$
there exists a subset $S_j\subseteq \{1,2,\dotsc,m\}$ such that
$\alpha(x_j)= \sum_{t\in S(j)} y_t$.
Furthermore, $S_i\cap S_j = \varnothing$ if $i\neq j$
and $\bigcup_{j=1}^k S_j = \{1,2,\dotsc,m\}$. So we have 
\[
\sum_{t\in S_j} \beta_2(y_t) = \beta_2(\alpha(x_j)) = \beta_1(x_j) =
\sum_{i=1}^n \beta_{1,i}(x_j).
\]
By the Riesz interpolation property, there exist $z_{i,t}\in H^{+}$
for $t\in S_j$ and $i=1,2,\dotsc,n$ so that
\[
\sum_{t\in S_j} z_{i,t} = \beta_{1,i}(x_j) \text{ and }
\sum_{i=1}^n z_{i,t} = \beta_2(y_t).
\]

We define $\beta_{2,i}\colon G_2\to H$ by setting $\beta_{2,i}(y_t)
= z_{i,t}$ for all $i,t$, setting $\beta_{2,1}(g) = \beta_2(g)$ for
$g\in \ker\rho_{G_2}$ and setting $\beta_{2,i}(g) = 0$ for
$g\in \ker\rho_{G_2}$ and $i>1$. We see that $\beta_{2,i}$ are
well-defined since the sets $S_j$ partition $\{1,2,\dotsc,m\}$ and
$\beta_{2,i}$ are group homomorphisms since atoms are $\Z$-independent.
By construction, $\ker\rho_{G_2}\subseteq \ker\beta_{2,i}$ for $i>1$.

As before, $\sum_{i=1}^n \beta_{2,i} = \beta_{2,1} = \beta_2$ on
$\ker\rho_{G_2}$ and $\sum_{i=1}^n \beta_{2,i}(y_t)
= \sum_{i=1}^n z_{i,t} = \beta_2(y_t)$.
So $\sum_{i=1}^n \beta_{2,i} = \beta_2$.

Notice that for all $\sigma\in S(G_2,u_2)$, we have
$\sigma\circ\alpha \in S(G_1,u_1)$, so if $g\in \ker\rho_{G_1}$, then
$\sigma\circ\alpha(g) = 0$ for all $\sigma\in S(G_2,u_2)$. So
$\alpha(g)\in \ker\rho_{G_2}$.
So on $\ker\rho_{G_1}$, $\beta_{1,i} = 0 = \beta_{2,i}\circ\alpha$
when $i>1$ and
$\beta_{1,1} = \beta_1 = \beta_2\circ\alpha = \beta_{2,1}\circ\alpha$.
Also $\beta_{2,i}(\alpha(x_j)) = \sum_{t\in S_j}\beta_{2,i}(y_t)
= \sum_{t\in S_j} z_{i,t} = \beta_{1,i}(x_j)$ for all $i,j$.
Thus $\beta_{1,i} = \beta_{2,i}\circ\alpha$ for all $i$.
Further, since $\alpha(u_1) = u_2$, we have $\beta_{2,i}(u_2)
= \beta_{1,i}(u_1) = v_i$.

Let $x\in G_2^{+}$. So there exist non-negative integers $r_t$ and
$g\in \ker\rho_{G_2}$ so that $x = g + \sum_{t=1}^m r_t y_t$. Take
$\tau\in T$. Since $\tau\circ\beta_{2,i}\in S(G,u)$, we see
$\tau(\beta_{2,i}(g)) = 0$ and so
$\tau(\beta_{2,i}(x)) = \sum_{t=1}^m r_t \tau(z_{i,t}) \geq 0$.
So $\beta_{2,i}$ are normalized positive group homomorphisms.

we can arrange it so that
$\ker\beta_{2,i} \cap \atom(G_2) = 0$ for all $i$.

Now assume $H$ has strict interpolation, that $\alpha(x_j)\neq 0$ for
$j=1,2,\dotsc,k$, that
$\ker\beta_2 \cap \atom(G_2) = 0$, and that
$\ker\beta_i \cap \atom(G_1) = 0$ for all $i$. 
It follows that all of the $S_j$ are non-empty, that
$\beta_2(y_t) > 0$ and $\beta_{1,i}(x_j) > 0$ for all $i,j,t$.
So from strict interpolation, we can arrange for $z_{i,t} > 0$. So
$\ker\beta_{2,i}\cap \atom(G_2) = 0$ for all $i$. 
\end{proof}

\section{Homomorphisms from $C(X)$ to \cstar-algebras of Tracial Rank
One} \label{sec:main}

The next lemma basically states that approximate diagonalization can
be reduced to the decomposition of compatible $K_0$ and trace maps.

\begin{lem} \label{lem:reduction}
Let $C$ be a unital stably finite \cstar-algebra and
let $A$ be a unital separable stably finite \cstar-algebras.
Assume we are given:
\begin{asparaenum}[1.]
\item a normalized positive group homomorphism
$\alpha\colon K_0(C)\to K_0(A)$,
\item a strictly positive unital linear map
$\gamma\colon \Aff(T(C))\to \Aff(T(A))$, 
\item an element $\kappa\in KL_e(C,A)^{++}$ such that
$\kappa$ restricted to $K_0(C)$ is $\alpha$, and
\item a group homomorphism
$\eta\colon \algK(C) \to \algK(A)$ such that 
$\kappa,\gamma,\eta$ are compatible.
\end{asparaenum}

Suppose there exist positive normalized group homomorphisms
$\alpha_i\colon K_0(C) \to K_0(e_iAe_i)$
and strictly positive linear maps
$\gamma_i\colon \Aff(C)\to \Aff(e_iAe_i)$
for $i=1,2,\dotsc, n$ such that $\alpha_i, \gamma_i$ are
compatible for $i=1,2,\dotsc,n$ and
\begin{align*}
\alpha &= \alpha_1 + \alpha_2 + \dotsb + \alpha_n, \text{ and} \\
\gamma &= \gamma_1 + \gamma_2 + \dotsb + \gamma_n.
\end{align*}

Then there exist elements $\kappa_i \in KL_e(C,e_iAe_i)^{++}$ and
continuous homomorphisms
$\eta_i\colon \algK(C) \to \algK(e_iAe_i)$ for $i=1,2,\dotsc,n$ such
that $\kappa_i, \gamma_i, \eta_i$ are compatible,
\begin{align*}
\kappa 	&= \kappa_1 + \kappa_2 + \dotsb + \kappa_n, \text{ and} \\
\eta	&= \eta_1 + \eta_2 + \dotsb + \eta_n.
\end{align*}
\end{lem}

\begin{proof}
Let $\beta\colon K_1(C)\to K_1(A)$ be the restriction of $\kappa$
to $K_1(C)$. We define group homomorphisms
$\beta_i\colon K_1(C)\to K_1(e_iAe_i)$
by
\[
\beta_i =
\begin{cases}
\beta 	& \text{ if } i = 1 \\
0 		& \text{ if } i \neq 1
\end{cases}
\]
for $i=1, 2, \dotsc, n$. So $\sum_{i=1}^n \beta_i = \beta_1 = \beta$.

For $1< i\leq n$, by the UCT, there exist $\kappa_i\in KL(C,e_iAe_i)$ such
that $\Gamma(\kappa_i) = (\alpha_i,\beta_i)$. We set
$\kappa_1 = \kappa - \sum_{i=2}^{n} \kappa_i$.
Notice that
$
\Gamma(\kappa_1) = (\alpha,\beta) - \sum_{i=2}^n (\alpha_i,\beta_i)
= (\alpha_1,\beta_1)$.
Since $\alpha_i$ is a positive, normalized group homomorphism,
compatible with $\gamma_i$, it follows that
$\kappa_i\in KL_e(C,e_iAe_i)^{++}$ is compatible with $\gamma_i$, and
by construction, $\kappa_1 + \kappa_2 + \dotsb + \kappa_n = \kappa$.

The compatible pair $(\kappa_i, \gamma_i)$ induces the group
homomorphism
\[
\eta_i^{0} \colon \Aff(T(C))/\rho_{C}(K_0(C)) \to
\Aff(T(e_iAe_i))/\rho_{e_iAe_i}(K_0(e_iAe_i)).
\]

We extend $\eta_i^{0}$ to a homomorphism
$\eta_i\colon \algK(C) \to \algK(e_iAe_i)$ by setting
\[
\eta_i(u) =
\begin{cases}
\eta(u)	& \text{ if } i=1 \\
0		& \text{ if } i\neq 1.
\end{cases}
\]
for $u\in K_1(C)$. By naturality, we have
$\pi_A\circ \eta_1 = \beta\circ \pi_{C} = \beta_1\circ\pi_{C}$,
and so $\kappa_1, \eta_1$ are compatible.
Since $\beta_i = 0 = \eta_i$ on $K_1(C)$ for $i=2,3,\dotsc,n$,
$\kappa_i, \eta_i$ are compatible for $i=2,3,\dotsc,n$.
By construction, $\gamma_i, \eta_i$ are compatible for
$i = 1,2,\dotsc, n$.
We see that the $\kappa_i, \gamma_i, \eta_i$ are compatible for
$i = 1,2,\dotsc, n$. Since $\eta_i$ restrict to $\beta_i$ on $K_1(C)$
and $\eta_i$ is induced from $\gamma_i$ on
$\Aff(T(C))/\rho_{C}(K_0(C))$, we have
$\eta_1 + \eta_2 + \dotsb + \eta_n = \eta$ on $\algK(C)$.
\end{proof}

\begin{thm} \label{thm:approxdiag-trone}
Let $X$ be a compact metric space. Let $A$ be a separable simple
unital \cstar-algebra with tracial rank at most one. 
Every unital homomorphism $\phi\colon C(X) \to A$
is approximately diagonalizable.
\end{thm}

\begin{proof}
By factoring out the kernel and applying the Gelfand-Naimark
theorem, we may assume, without loss of generality, that $\phi$ is
injective.

Let $\varepsilon > 0$ and let $\F\subseteq C(X)$ be a finite subset. By
Theorem \ref{thm:trone-unique}, there exist $\delta > 0$, a finite subset
$\G\subseteq C(X)$, a finite subset $\P\subseteq \uK(C(X))$, and a
finite subset $\U\subseteq U^{\infty}(C(X))$ such that for any unital
homomorphism $\psi\colon C(X)\to A$, if
\begin{asparaenum}[1.]
\item $KL(\phi) = KL(\psi)$ on $\P$,
\item $\dist(\phi^{\ddagger}(\bar{z}),\psi^{\ddagger}(\bar{z})) < \delta$
for $z\in \U$, and
\item $\abs{\tau\circ\phi(g) - \tau\circ\psi(g)} < \delta$ for $g\in\G$,
\end{asparaenum}
then there exists a unitary $u\in A$ such that
\[
\norm{ u\phi(f)u^{*} - \psi(f) } < \varepsilon
\]
for all $f\in \F$. 

Since $X$ is a compact metric space, there exist finite simplicial
complexes $X_m$ for $m\in \N$ and unital homomorphisms
$s_m \colon C(X_m) \to C(X_{m+1})$ such that
$C(X) \cong \varinjlim C(X_m)$. 
Let $s_{m,\infty}\colon C(X_m) \to C(X)$ denote the homomorphisms
induced by the inductive limit. Let $k(m)$ denote the number of
connected components of $X_m$ and let $\chi_m^j$ the characteristic
functions of the connected components of $X_m$ for $j=1,2,\dotsc,k(m)$.
We may assume that $s_{m}(\chi_m^j)\neq 0$ for all $j$.

Since $\G$ is finite, there exist an integer $M$ and a finite set
$\G'\subseteq C(X_M)_{\sa}$ such that for every $g\in \G$, there exists
$g'\in \G'$ such that $\norm{g - s_{M,\infty}(g')} < \delta/2$.

Furthermore, by taking a possibly larger value of $M$, there exists a
finite set $\U'\subseteq U^{\infty}(C(X_M))$ such
that for every $u\in \U$, there exists $u_0\in \U'$ such that
$\dist(\bar{u},s_{M,\infty}^{\ddagger}(\bar{u}_0)) < \delta/2$.

Since $X_M$ has finitely many connected components, $C(X_M,\Z)$ is
generated by the atoms of $K_0(C(X_M))_{+}$ and so
\[
K_0(C(X_M)) = C(X_M,\Z)\oplus \ker\rho_{C(X_M)}.
\]
In addition we see that since $\phi$ is injective, 
$\ker K_0(\phi)\cap C(X,\Z) = 0$. So by Lemma \ref{lem:gp-base}, there
exist normalized group homomorphisms
$\alpha_{i,M}\colon (K_0(C(X_M)),1_{C(X_M)}) \to (K_0(A),[e_i])$ such that
$\ker\alpha_{i,M} \cap C(X_M,\Z) = 0$ for all $i$, 
$\ker\alpha_{i,M} = \ker\rho_{C(X_M)}$ when $i>1$ and
\[
K_0(\phi\circ s_{M,\infty}) = \sum_{i=1}^n \alpha_{i,M}.
\]

Since $A$ has stable rank one, there exist non-zero mutually orthogonal
projections $p^M_{i,j} \in A$ for $i=1,2,\dotsc,n$ and
$j=1,2,\dotsc,k(M)$ such that
$[p^M_{i,j}] = \alpha_{i,M}(\chi_M^j)$
and 
\[
\phi(s_{M,\infty}(\chi_M^j)) = \sum_{i=1}^n p^M_{i,j}. 
\]
We define
$\gamma_{i,M}\colon C(X_M)_{\sa} \to \Aff(T(A))$ by
\[
\gamma_{i,M}(f)(\tau)
= \sum_{j=1}^{k(M)} \tau(p^M_{i,j} \phi\circ s_{M,\infty}(f) p^M_{i,j}).
\]
Since the projections $p^M_{i,j}$ are non-zero and mutually orthogonal,
$\gamma_{i,M}$ is a positive, linear map with $\ker\gamma_{i,M} =
\ker s_{M,\infty}$. For all $\tau\in T(A)$ and $j_0$, we have
\[
\gamma_{i,M}(\chi_M^{j_0})(\tau)
=\sum_{j=1}^{k(M)}\tau(p^M_{i,j}\phi(s_{M,\infty}(\chi_M^{j_0}))p^M_{i,j})
=\tau(p^M_{i,j_0}) = \tau(\rho_A(\alpha_i(\chi_M^{j_0}))).
\]
So $\alpha_{i,M}, \gamma_{i,M}$ are compatible for $i=1,2,\dotsc,n$.

We inductively apply Lemma \ref{lem:gp-induction} to construct normalized
positive group homomorphisms $\alpha_{i,m}\colon K_0(C(X_m))\to K_0(A)$
for $i=1,2,\dotsc,n$ and $m\geq M$ so that
$K_0(\phi\circ s_{m,\infty}) = \sum_{i=1}^n\alpha_{i,m}$ with
$\alpha_{i,m} = \alpha_{i,m+1}\circ s_m$, and
$\ker\alpha_{i,m} \cap C(X_m,\Z) = 0$ for all $i$ with
$\ker\alpha_{i,m} = \ker\rho_{C(X_m)}$ when $i>1$.

As before, there exist non-zero mutually orthogonal projections
$p^m_{i,j}\in M_n(A)$ for $i=1,2,\dotsc,n$ and $j=1,2,\dotsc,k(m)$ such
that $[p^m_{i,j}] = \alpha_{i,m}(\chi_m^j)$ and
\[
\phi\circ s_{m,\infty}(\chi_m^j) = \sum_{i=1}^n p^m_{i,j}. 
\]
We see that $\gamma_{i,m}$ is a positive unital linear map with
$\ker\gamma_{i,m} = \ker s_{m,\infty}$. The
pair $\alpha_{i,m}, \gamma_{i,m})$ is compatible by a computation
identical to the case where $m=M$.

Let $\alpha_i$ be the homomorphism induced by the inductive limit and
the homomorphisms $\alpha_{i,m}$ and let $\gamma_i$ be the linear map
induced by the inductive limit and the linear maps $\gamma_{i,m}$.
Since
\[
K_0(\phi\circ s_{m,\infty}) =
\alpha_{1,m} + \alpha_{2,m} + \dotsb + \alpha_{n,m},
\]
by the uniqueness maps induced by the inductive limit, we have
\[
K_0(\phi) = \alpha_1 + \alpha_2 + \dotsb + \alpha_n. 
\]
Since $\ker\alpha_{i,m}\cap C(X_m,\Z) = 0$, it follows that
$\ker\alpha_i \cap C(X,\Z) = 0$ for all $i$. Also $\gamma_i$ is
injective, since $\ker\gamma_{i,m} = \ker s_{m,\infty}$. And since
$\alpha_{i,m},\gamma_{i,m}$ are compatible, we have that
$(\alpha_i, \gamma_i)$ are compatible.

By Lemma \ref{lem:reduction}, there exist
$\kappa_i\in KL_e(C(X),e_iAe_i)^{++}$ such that
$\kappa_1 + \kappa_2 + \dotsb + \kappa_n = KL(\phi)$
and homomorphisms
$\eta_i\colon \algK(C(X)) \to \algK(e_iAe_i)$ such that
$\eta_1 + \eta_2 + \dotsb + \eta_n = \phi^{\ddagger},$
and such that $\kappa_i, \gamma_i, \eta_i$ are compatible for
$i = 1, 2, \dotsc, n$.

We note that
\[
\sum_{i=1}^n \eta_i\circ s_{M,\infty}^{\ddagger} =
(\phi\circ s_{M,\infty})^{\ddagger}
\]
on $\algK(C(X_M))$.

By Theorem 4.5 of \cite{lin-trone}, there exist unital monomorphisms
$\phi_i\colon C(X)\to e_iAe_i$ such that
\begin{align*}
KL(\phi_i) &= \kappa_i, \\
\tau(\phi_i(f)) &= \gamma_i(f)(\tau), \text{ and} \\
\phi_i^{\ddagger} &= \eta_i.
\end{align*}
for all $f\in C(X)_{\sa}$ and $\tau\in T(A)$. Let
$\psi = \sum_{i=1}^n \phi_i$ 
So
\[
KL(\psi) = \sum_{i=1}^{n}KL(\phi_i) = \sum_{i=1}^{n}\kappa_i
= KL(\phi).
\]
In particular, this holds for $\P$.

Let $f\in \G$ and $\tau \in T(M_n(A))$. There exists $f'\in \G'$
so that $\norm{f - s_{M,\infty}(f')} < \delta/2$. Note that 
\begin{align*}
\tau(\psi(s_{M,\infty}(f'))) & =
\sum_{i=1}^{n}\gamma_i(s_{M,\infty}(f'))(\tau) \\ &=
\sum_{i=1}^{n}\sum_{j=1}^{k(M)}\tau(p^M_{i,j}
\phi(s_{M,\infty}(f')) p^M_{i,j}) \\ & =
\sum_{j=1}^{k(M)} \tau(\phi(s_{M,\infty}(\chi_M^j))
\phi(s_{M,\infty}(f')) \phi(s_{M,\infty}(\chi_M^j)))
\\ &= \tau(\phi(s_{M,\infty}(f'))).
\end{align*}
Consequently,
\begin{align*}
\abs{\tau(\phi(f)) - \tau(\psi(f))} & \leq 
\abs{\tau(\phi(f)) - \tau(\phi(s_{M,\infty}(f')))} \\
& \phantom{\leq} + \abs{\tau(\phi(s_{M,\infty}(f')))
  \phantom{\leq} - \tau(\psi(s_{M,\infty}(f')))} \\
& \phantom{\leq} + \abs{\tau(\psi(s_{M,\infty}(f'))) - \tau(\psi(f))} \\
& < \norm{\tau\circ\phi} \left(\delta/2\right) +
\norm{\tau\circ\psi} \left(\delta/2\right) \\
& = \delta.
\end{align*}

Let $u\in \U$. There exists $u_0\in \U'$ such that
$\dist(\bar{u},s_{M,\infty}^{\ddagger}(\bar{u}_0)) < \delta/2$.
So we have
\begin{align*}
\dist(\phi^{\ddagger}(\bar{u}),\psi^{\ddagger}(\bar{u})) & \leq
\dist(\phi^{\ddagger}(\bar{u}),(\phi\circ s_{M,\infty})(\bar{u})) \\
& \phantom{\leq} +
\dist((\phi\circ s_{M,\infty})^{\ddagger}(\bar{u}),
(\psi\circ s_{M,\infty})^{\ddagger}(\bar{u})) \\
& \phantom{\leq} +
\dist((\psi\circ s_{M,\infty})^{\ddagger}(\bar{u}),
\psi^{\ddagger}(\bar{u})) \\ & \leq
\delta/2 + 0 + \delta/2 \\
&= \delta. \qedhere
\end{align*}
\end{proof}

\section{Other Homomorphisms} \label{sec:other}

Due to a similar classification of homomorphisms from AH-algebras to
\cstar-algebras with rational tracial rank one (see \cite{lin-niu-cl}),
one might expect a similar result about approximate diagonalization.
But with one notable exception, namely when the $K_0$-group is cyclic,
none of the \cstar-algebras in this expanded class have the Riesz
interpolation property, and the Riesz interpolation property for the
$K_0$-group is a necessary condition for approximate diagonalization.
In particular, if $A$ has stable rank one and $K_0(A)$ is not an
interpolation group, then there exists a projection in $A$ that is not
approximately diagonalizable. More generally, if $A$ is stably finite
and $K_0(A)$ is not an interpolation group, then there exists a positive
integer $n\geq 1$ and a projection in $M_n(A)$ that is not approximately
diagonalizable.

\begin{thm} \label{thm:cyclic}
Let $X$ be a compact metric space such that $K_1(C(Y))$ is free for
every compact subset $Y\subseteq X$.
Let $A$ be a separable simple unital \jiangsu-stable \cstar-algebra
with rational tracial rank at most one such that $K_0(A) = \Z$.
Every unital homomorphism $\phi\colon C(X)\to A$ is
approximately diagonalizable.
\end{thm}

\begin{proof}
As before, we assume that $\phi$ is injective. This implies that $X$
has finitely many connected components, since otherwise there would
exist infinitely many positive integers with a finite sum, which is
absurd. Let $X_j$ denote the connected components of $X$ for
$j=1,\dotsc,k$ and let $\chi_j$ denote the characteristic function of
$X_j$.

Since $\sum_j [\phi(\chi_j)] = \sum_i [e_i]$, by interpolation, there
exist elements $z_{i,j} \in K_0(e_iAe_i)^{+}$ such that
\[
\sum_{i=1}^n z_{i,j} = [\phi(\chi_j)] \text{ and }
\sum_{j=1}^k z_{i,j} = [e_i].
\]
Let $S_i = \{ j\colon z_{i,j} \neq 0\}$. Let $m_j =
\min\{ i\colon j\in S_i\}$. Let $Y_i = \cup_{j\in S_i} X_j$. 

We define $\alpha_i\colon K_0(C(Y_i)) \to K_0(e_iAe_i)$ by
$\alpha_i(\chi_j) = z_{i,j}$ and 
\[
\alpha_i(g) = 
\begin{cases}
[\phi(g)] & \text{ if } i = m_j \\
0 & \text{ otherwise}
\end{cases}
\]
for $g\in \ker\rho_{C(X_j)}$. Let
$\sum_{j=1} (c_j\chi_j + g_j)\in K_0(C(X))$ be an arbitrary element,
where $c_j\in \Z$ and $g_j\in \ker\rho_{C(X_j)}$. It is not difficult
to see that $\sum_{i=1}^n \alpha_i = [\phi]$.

Note that
\[
\alpha_i(1)  = \sum_{j\in S_i} \alpha_i(z_{i,j})
=  \sum_{j=1}^k \alpha_i(z_{i,j}) = [e_i].
\]

Since $A$ has stable rank one, there exist mutually orthogonal, non-zero
projections $p_{i,j}\in A$ for $i=1,\dotsc,n$ and $j\in S_i$ such that
$\sum_{j\in S_i} p_{i,j} = 1$ and $[p_{i,j}] = z_{i,j}$. Let
$\gamma_i\colon C(Y_i)_{\sa} \to \Aff(T(e_iAe_i))$ be defined by
$\gamma_i(f)(\tau) = \sum_{j\in S_i} \tau(p_{i,j}fp_{i,j})$. Similar to
before, $\gamma_i$ is a strictly positive linear map compatible with
$\alpha_i$.

By Lemma \ref{lem:reduction}, there exist elements
$\kappa_i\in KL_e(C(Y_i),e_iAe_i)^{++}$ such that
\[
\kappa_1 + \kappa_2 + \dotsb + \kappa_n = KL(\phi)
\]
and group homomorphisms
$\eta_i\colon \algK(C(Y_i)) \to \algK(e_iAe_i)$ 
such that
\[
\eta_1 + \eta_2 + \dotsb + \eta_n = \phi^{\ddagger}
\]
and $\kappa_i, \gamma_i, \eta_i$ are compatible for
$i=1,2,\dotsc,n$.

So by Theorem 6.10 of \cite{lin-niu-cl}, there exist unital
homomorphisms $\psi_i^0\colon C(Y_i)\to e_iAe_i$ for $i=1,2,\dotsc,n$
such that
\begin{align*}
KL(\psi^0_i) &= \kappa_i, \\
\tau(\psi^0_i(f)) &= \gamma_i(f)(\tau), \text{ and} \\
(\psi^0_i)^{\ddagger} &= \eta_i.
\end{align*}
Since $C(X) = C(Y_i) \oplus C(Y_i^c)$, we can extend $\psi_i^0$ by
setting $\psi_i^0(f) = 0$ for $f\in C(Y_i^c)$.

Let $\psi = \sum_{i=1}^n \psi_i$. We can see that
\begin{align*}
KL(\psi) &= \sum_{i=1}^{n} KL(\phi_i) = \sum_{i=1}^n \kappa_i =
KL(\phi), \\
\tau(\psi(f)) &= \sum_{i=1}^n \tau(\phi_i(f)) =
\sum_{i=1}^n \gamma_i(f)(\tau) = \tau(\phi(f)), \\
\text{and }
\psi^{\ddagger}
&= \sum_{i=1}^n \phi_i^{\ddagger}
= \sum_{i=1}^n \eta_i =
\phi^{\ddagger}.
\end{align*}
So by Corollary 5.4 of \cite{lin-niu-cl}, $\phi$ and $\psi$ are
approximately unitarily equivalent.
\end{proof}

Also, in the case where the spectrum is connected, we do not require
the Riesz interpolation property.

\begin{thm} \label{thm:rationalconnected}
Let $X$ be a compact connected metric space such that $K_1(C(X))$
is free. Let $A$ be a simple separable unital \jiangsu-stable
\cstar-algebra with rational tracial rank at most one.
Every unital homomorphism $\phi\colon C(X)\to A$ is approximately
diagonalizable.
\end{thm}

\begin{proof}
Since $X$ is connected, $C(X,\Z)\cong \Z$ and so
$K_0(C(X)) = \Z \oplus \ker\rho_{C(X)}$.
Also $[1_{C(X)}] = (1,0)$ in this decomposition.
We define normalized group homomorphisms
$\alpha_i\colon K_0(C(X)) \to K_0(e_iAe_i)$ by
$\alpha_i(1_{C(X)}) = [e_i]$ on $C(X,\Z)$ and
\[
\alpha_i =
\begin{cases}
K_0(\phi) 	& \text{ if } i = 1 \\
0		& \text{ if } i \neq 1
\end{cases}
\]
on $\ker\rho_{C(X)}$. One can readily see that
$\alpha_1 + \alpha_2 + \dotsb + \alpha_n = K_0(\phi)$.

We define $\gamma_i\colon C(X)_{\sa}\to \Aff(T(e_iAe_i))$ by
$\gamma_i(f)(\tau) = \tau(e_i\phi(f)e_i)$ for $i=1,2,\dotsc,n$.
Since $\rho_{C(X)}(C(X))$ is cyclic and $\gamma_i$ is unital,
$(\alpha_i, \gamma_i)$ are compatible for $i=1,2,\dotsc,n$.

By Lemma \ref{lem:reduction}, there exist elements
$\kappa_i\in KL_e(C(X),e_iAe_i)^{++}$ such that
\[
\kappa_1 + \kappa_2 + \dotsb + \kappa_n = KL(\phi)
\]
and group homomorphisms
$\eta_i\colon \algK(C(X)) \to \algK(e_iAe_i)$ 
such that
\[
\eta_1 + \eta_2 + \dotsb + \eta_n = \phi^{\ddagger}
\]
and $\kappa_i, \gamma_i, \eta_i$ are compatible for
$i=1,2,\dotsc,n$.

So by Theorem 6.10 of \cite{lin-niu-cl}, there exist unital
homomorphisms $\psi_i\colon C(X)\to e_iAe_i$ for $i=1,2,\dotsc,n$ such
that
\begin{align*}
KL(\psi_i) &= \kappa_i, \\
\tau(\psi_i(f)) &= \gamma_i(f)(\tau), \text{ and} \\
\psi_i^{\ddagger} &= \eta_i.
\end{align*}

Let $\psi = \sum_{i=1}^n \psi_i$. We can see that
\begin{align*}
KL(\psi) &= \sum_{i=1}^{n} KL(\phi_i) = \sum_{i=1}^n \kappa_i =
KL(\phi), \\
\tau(\psi(f)) &= \sum_{i=1}^n \tau(\phi_i(f)) =
\sum_{i=1}^n \gamma_i(f)(\tau) = \tau(\phi(f)), \\
\text{and }
\psi^{\ddagger}
&= \sum_{i=1}^n \phi_i^{\ddagger}
= \sum_{i=1}^n \eta_i =
\phi^{\ddagger}.
\end{align*}
So by Corollary 5.4 of \cite{lin-niu-cl}, $\phi$ and $\psi$ are
approximately unitarily equivalent.
\end{proof}

We can also consider more general $AH$-algebras for the domains of the
homomorphisms instead of commutative \cstar-algebras. But there are few
cases where we have general results.
But even when restricted to the case of homomorphisms between
AF-algebras, approximate diagonalization becomes more difficult to
analyze.

\begin{thm} \label{thm:uniquetrace}
Let $C$ be a separable unital AH-algebra with unique tracial state and
let $A$ be a separable simple unital \cstar-algebra with tracial rank
at most one. Every unital homomorphism $\phi\colon C \to A$ is
approximately diagonalizable if for every projection $p$, there exists a
unital homomorphism $\phi\colon C\to pAp$.
\end{thm}

\begin{proof}
Since $C$ is exact, $K_0(C)$ has a unique trace. For $i>1$, by
assumption, there exists a positive group homomorphism
$\alpha_i\colon K_0(C)\to K_0(e_iAe_i)$ such that
$\alpha_i(1_C) = [e_i]$. Let
$\alpha_1 = [\phi] - \sum_{i=2}^n \alpha_i$. We wish to show $\alpha_1$
is positive. Let $\sigma$ denote the unique trace of $K_0(C)$. So given
a positive non-zero element $g\in K_0(C)^{+}$, we have $\sigma(g) > 0$.
Let $\tau$ be a trace on $K_0(A)$. Since
$\tau\circ[\phi]$ and $\tau(e_i)^{-1}(\tau\circ\alpha_i)$ are traces on
$K_0(C)$, we see that $\tau\circ[\phi] = \sigma$ and
$\tau\circ\alpha_i = \tau(e_i)\sigma$. So
\begin{align*}
\tau(\alpha_1(g)) &= \tau([\phi(g)]) - \sum_{i=2}^n \tau(\alpha_i(g))
\\ &= \sigma(g) - \sum_{i=2}^n \tau(e_i)\sigma(g) =
\tau(e_1)\sigma(g) > 0.
\end{align*}

So $\alpha_1(g) > 0$. So $\alpha_i$ is a normalized positive group
homomorphism for all $i$. Since $\Aff(T(C)) \cong \R$, there exists a
unique normalized positive linear map $\gamma_i\colon \Aff(T(C))\to
\Aff(T(e_iAe_i))$ which is compatible with $\alpha_i$ since $C$ is
exact.

By Lemma \ref{lem:reduction} and Theorem 4.5 of \cite{lin-trone}, there
exist unital homomorphisms $\psi_i\colon C\to e_iAe_i$ such that
$\sum_{i=1}^n KL(\psi_i) = KL(\phi)$, $\tau\circ\phi =
\sum_{i=1}^n \tau\circ\psi_i$ for $\tau\in T(C)$, and
$\sum_{i=1}^n \psi_i^{\ddagger} = \phi^{\ddagger}$. So by Corollary 5.4
of \cite{lin-niu-cl}, $\phi$ and $\sum_i \psi_i$ are approximately
unitarily equivalent.
\end{proof}


For a concrete example of the complications that arise, let $G_0$ denote
the subgroup of $\R^2$ generated by $(1,0)$, $(0,1)$ and
$(\sqrt{2},\sqrt{3})$ with order induced from the strict ordering on
$\R^2$. Let $u = (1,1)$. Let $H_0$ denote the subgroup of $\R$ generated
by $1$ and $\sqrt{2} + \sqrt{3}$. By the Effros-Handelman-Shen Theorem
(see Theorem 2.2 of \cite{effros-handelman-shen}), there exist unital
simple AF-algebras $C_0$ and $A_0$ such that $K_0(G_0) = C_0$ and
$K_0(A_0) = H_0$. Let $\Hom_c(G_0,H_0)$ denote the set of group
homomorphisms that ``extend'' to linear maps from $\Aff(S(G_0))$ to
$\Aff(S(H_0))$. There is a one-to-one correspondence between
$\Hom_c(G_0,H_0)$ and $\Z^2$  given by
$\alpha \mapsto (\alpha(1,0),\alpha(0,1))$. Furthermore, $\alpha$ is
positive if and only if the corresponding lattice point $(x,y)$
satisfies $y\geq \pm x(\sqrt{2}+\sqrt{3})$. Finally, let $v_1, v_2$ be
two positive even integers and let $p_1$ and $p_2$ be two projections
in $A$ with $[p_i] = v_i$. Every unital homomorphism from $C_0$ to
$A_0$ (where $[1_{A_0}] = v_1 + v_2$) is approximately diagonalizable
with respect to $p_1$ and $p_2$ if and only if
\[
\floor{\frac{v_1}{2(\sqrt{2}+\sqrt{3}}} +
\floor{\frac{v_2}{2(\sqrt{2}+\sqrt{3}}} =
\floor{\frac{v_1}{2(\sqrt{2}+\sqrt{3}} +
\frac{v_2}{2(\sqrt{2}+\sqrt{3}}}.
\]

Put geometrically, approximate diagonalization is equivalent to an
equation involving sumsets of certain subsets of cones of lattice
points.

\bibliographystyle{plain}
\bibliography{diagonal}

\begin{thebibliography}{10}

\bibitem{umct}
Marius Dadarlat and Terry~A. Loring.
\newblock A universal multi-coefficient theorem for the {Kasparov} groups.
\newblock {\em Duke Math. J.}, 84:355--377, 1996.

\bibitem{effros-handelman-shen}
Edward~G. Effros, David~E. Handelman, and Chao~Liang Shen.
\newblock Dimension groups and their affine representations.
\newblock {\em Amer. J. Math.}, 102:355--377, 1996.

\bibitem{goodearl}
Kenneth~R. Goodearl.
\newblock {\em Partially Ordered Abelian Groups with Interpolation}.
\newblock Amer. Math. Soc., Rhode Island, USA, 1986.

\bibitem{grpd-diag}
Karsten Grove and Gert~K. Pedersen.
\newblock Diagonalizing matrices over {$C(X)$}.
\newblock {\em J. Funct. An.}, 59:65--89, 1984.

\bibitem{here-diag}
Chris Heunen and Manuel~L. Reyes.
\newblock Diagonalizing matrices over {$AW^{*}$}-algebras.
\newblock {\em J. Funct. An.}, 264:1873--1898, 2013.

\bibitem{kad-diag}
Richard~V. Kadison.
\newblock Diagonalizing matrices.
\newblock {\em Amer. J. Math.}, 106:1451--1468, 1984.

\bibitem{lin-auetrone}
Huaxin Lin.
\newblock Approximate unitary equivalence in simple \cstar-algebras of tracial
  rank one.
\newblock {\em Trans. Amer. Math. Soc.}, 264:2021--2086, 2012.

\bibitem{lin-diag}
Huaxin Lin.
\newblock Approximately diagonalizing matrices over {$C(Y)$}.
\newblock {\em Proc. Natl. Acad. Sci. U.S.A.}, 109:2842--2847, 2012.

\bibitem{lin-trone}
Huaxin Lin.
\newblock Homomorphisms from {AH}-algebras.
\newblock arXiv, 2013.

\bibitem{lin-niu-cl}
Huaxin Lin and Zhuang Niu.
\newblock Homomorphisms into simple \jiangsu-stable \cstar-algebras.
\newblock {\em J. Operator Theory}, 71:517--569, 2012.

\bibitem{ror-cl}
Mikael~R\o rdam.
\newblock {\em Classification of Nuclear \cstar-Algebras}.
\newblock Springer, New York, USA, 2002.

\bibitem{thom-traces}
Klaus Thomsen.
\newblock Traces, unitary characters and crossed products by ${\Z}$.
\newblock {\em Publ. Res. Inst. Math. Sci.}, 31:1011--1029, 1995.

\bibitem{xue-diag}
Yifeng Xue.
\newblock Approximate diagonalization of self-adjoint matrices over {$C(X)$}.
\newblock {\em Funct. Anal. Approx. Comput.}, 2:53--65, 2010.

\bibitem{zh-diag}
Shuang Zhang.
\newblock Diagonalizing projections in multiplier algebras in matrices over a
  \cstar-algebra.
\newblock {\em Pacific J. Math.}, 145:181--200, 1990.

\end{thebibliography}

\end{document}